\documentclass[a4paper]{amsart}
\usepackage{amsmath,amsthm,amssymb}
\usepackage[british]{babel}
\usepackage[pdfborder={0 0 0}]{hyperref}
\usepackage{enumerate}

\newtheorem{thm}{Theorem}[section]

\newtheorem{lem}[thm]{Lemma}
\newtheorem{prop}[thm]{Proposition}

\newtheorem{question}[thm]{Question}
\theoremstyle{definition}
\newtheorem{defi}[thm]{Definition}

\DeclareFontFamily{T1}{pzc}{} 
\DeclareFontShape{T1}{pzc}{m}{it}{<-> s * [1.15] pzcmi8t}{} 
\DeclareMathAlphabet{\mathpzc}{T1}{pzc}{m}{it} 

\newcommand{\mm}{\mathpzc{M}}
\newcommand{\dd}{\mathpzc{D}}

\renewcommand{\aa}{\mathpzc{A}}
\newcommand{\A}{\mathpzc{A}}

\newcommand{\oo}{\mathcal{O}}
\newcommand{\B}{\mathpzc{B}}
\renewcommand{\phi}{\varphi}

\newcommand{\mw}{{\mm_w}}
\newcommand{\dom}{\mathrm{dom}}
\newcommand{\Th}{\mathrm{Th}}

\newcommand{\ba}{\mathbf{a}}
\newcommand{\bb}{\mathbf{b}}
\newcommand{\bc}{\mathbf{c}}
\newcommand{\bzero}{\mathbf{0}}

\newcommand{\M}{\mathpzc{M}}
\newcommand{\F}{\mathpzc{F}}

\begin{document}

\title{Natural Factors of the Muchnik Lattice Capturing IPC}
\author[R. Kuyper]{Rutger Kuyper}
\address[Rutger Kuyper]{Radboud University Nijmegen\\
Department of Mathematics\\
P.O. Box 9010, 6500 GL Nijmegen, the Netherlands.}
\email{r.kuyper@math.ru.nl}
\thanks{Research supported by NWO/DIAMANT grant 613.009.011 and by 
John Templeton Foundation grant 15619: `Mind, Mechanism and Mathematics: Turing Centenary Research Project'.}
\date{\today}
\maketitle

\begin{abstract}
We give natural examples of factors of the Muchnik lattice which capture intuitionistic propositional logic (IPC), arising from the concepts of lowness, 1-genericity, hyperimmune-freeness and computable traceability. This provides a purely computational semantics for IPC.
\end{abstract}

\section{Introduction}\label{sec-intro}

Ever since the introduction of intuitionistic logic by Heyting, there have been investigations into the computational content of proofs in intuitionistic logic. The best known of these is the investigation into realisability, which was initiated by Kleene in his 1945 paper \cite{kleene-1945}. Unfortunately, Kleene's original concept of realisability turns out to capture a proper extension of intuitionistic propositional logic (IPC). Nowadays, this field investigates not only Kleene realisability, but also many variations thereof; see e.g.\ the recent reference Van Oosten \cite{vanoosten-2008}.

Another approach to capture IPC in a computational way was provided by Medvedev \cite{medvedev-1955} and Muchnik \cite{muchnik-1963} in respectively 1955 and 1963. Their approaches, in the form of the \emph{Medvedev lattice} and the \emph{Muchnik lattice}, again turn out to fall short: they realise the weak law of the excluded middle $\neg p \vee \neg\neg p$. However, the study of these lattices did not end here, for multiple reasons.

First, the Medvedev and Muchnik lattices can be seen as generalisations of the Turing degrees (in fact, the Turing degrees can be embedded into both of these lattices). Therefore these lattices are of independent interest to computability theorists, regardless of any logical content they might carry. Research in this direction has increased in recent years; many details can be found in the surveys of Sorbi \cite{sorbi-1996} and Hinman \cite{hinman-2012}.

Furthermore, even on the logical side not all is lost: it turns out that we can repair the logical deficiency of the Medvedev and Muchnik lattices (i.e.\ the fact that they realise more than IPC). In \cite{skvortsova-1988}, Skvortsova shows that there is a factor of the Medvedev lattice which exactly captures IPC, and in Sorbi and Terwijn \cite{sorbi-terwijn-2012} the analogous result for the Muchnik lattice is shown.

These factors are obtained by taking the Medvedev or Muchnik lattice modulo a principal filter generated by some set $A$.  If we want to capture IPC in a truly computational way, we would want such a set $A$ to have some computational interpretation. Unfortunately, this is not the case for the sets $A$ appearing in the result of Skvortsova and in the result of Sorbi and Terwijn: instead of starting with some computationally motivated set $A$ and proving that the factor induced by this set $A$ captures IPC, they construct a set $A$ which exactly has the properties they require for their proof, but which does not seem to have any computational interpretation. In \cite{terwijn-2006}, Terwijn asked if there are any natural sets $A$ for which the factor captures IPC.

In the present paper, we will show that for the Muchnik lattice it is indeed possible to choose the set $A$ in a natural way (in the sense that it is definable using commonly used concepts from computability theory) and still obtain IPC as the theory of its factor. This way, we obtain a purely computational semantics for IPC. Aside from this, our results also put the computability-theoretic concepts used to define $A$ into a new light. Among these concepts are lowness, 1-genericity below $\emptyset'$, hyperimmune-freeness and computable traceability. Since our framework is general, our results could be adapted to suit other concepts.

In the next section we will briefly recall the structure of the Muchnik lattice and its factors. In section \ref{sec-splitting} we will describe our framework of \emph{splitting classes}. In section \ref{sec-low} we show that our framework is non-trivial by proving that the low functions and the functions of 1-generic degree below $\emptyset'$ fit in our framework. Next, in section \ref{sec-ipc} we prove that splitting classes naturally induce a factor of the Muchnik lattice which captures IPC. Finally, in section \ref{sec-hif} we consider whether two other concepts from computability theory give us splitting classes: hyperimmune-freeness and computable traceability.

Our notation is mostly standard. We let $\omega$ denote the natural numbers and $\omega^\omega$ the Baire space of functions from $\omega$ to $\omega$. For finite strings $\sigma,\tau$ we denote by $\sigma \subseteq \tau$ that $\sigma$ is a substring of $\tau$, by $\sigma \subset \tau$ that $\sigma$ is a proper substring of $\tau$ and by $\sigma \mid \tau$ that $\sigma$ and $\tau$ are incomparable. The concatenation of $\sigma$ and $\tau$ is denoted by $\sigma \star \tau$; for $n \in \omega$ we denote by $\sigma \star n$ the concatenation of $\sigma$ with the string $\langle n \rangle$. We assume a fixed, computable enumeration of the set of all finite binary strings. We let $\emptyset'$ denote the halting problem. By $\{e\}^A(n)[m] \downarrow$ we mean that the $e$th Turing machine with oracle $A$ and input $n$ terminates in at most $m$ steps. For functions $f,g \in \omega^\omega$ we denote by $f \oplus g$ the join of the functions $f$ and $g$, i.e.\ $(f \oplus g)(2n) = f(n)$ and $(f \oplus g)(2n+1) = g(n)$. For a poset $(X,\leq)$ and elements $x,y \in X$, we denote by $[x,y]_X$ the set of elements $u \in X$ satisfying $x \leq u \leq y$. For any set $\A \subseteq \omega^\omega$ we denote by $\overline{\A}$ its complement in $\omega^\omega$. When we say that a set is countable, we include the possibility that it is finite. For unexplained notions from computability theory, we refer to Odifreddi \cite{odifreddi-1989}, for the Muchnik and Medvedev lattices, we refer to the surveys of Sorbi \cite{sorbi-1996} and Hinman \cite{hinman-2012} (but we use the notation from Sorbi and Terwijn \cite{sorbi-terwijn-2012}), for lattice theory, we refer to Balbes and Dwinger \cite{balbes-dwinger-1974}, and finally for unexplained notions about Kripke semantics we refer to Chagrov and Zakharyaschev \cite{chagrov-zakharyaschev-1997}.

\section{Muchnik lattice and Brouwer algebras}\label{sec-muchnik}

We begin by briefly recalling the definition of and some elementary facts about the Muchnik lattice.

\begin{defi}{\rm(Muchnik \cite{muchnik-1963})}
Let $\aa,\B \subseteq \omega^\omega$ (we will call such subsets of $\omega^\omega$ \emph{mass problems}). We say that $\aa$ \emph{Muchnik reduces to} $\B$ (notation: $\aa \leq_w \B$) if for every $g \in \B$ there exists an $f \in \aa$ such that $f \leq_T g$. If $\aa \leq_w \B$ and $\B \leq_w \aa$ we say that $\aa$ and $\B$ are \emph{Muchnik equivalent} (notation: $\aa \equiv_w \B$). The equivalence classes of Muchnik equivalence are called \emph{Muchnik degrees} and the set of Muchnik degrees is denoted by $\M_w$.
\end{defi}

To avoid confusion, we do not use $\vee$ for the join (least upper bound) or $\wedge$ for the meet (greatest lower bound) in lattices, because later on we will see that the join corresponds to the logical conjunction $\wedge$ and that the meet corresponds to the logical disjunction $\vee$. Instead, we use $\oplus$ for join and $\otimes$ for meet.

\begin{defi}{\rm(McKinsey and Tarski \cite{mckinsey-tarski-1946})}
A \emph{Brouwer algebra} is a bounded distributive lattice together with a binary \emph{implication operator} $\to$ satisfying:
\[a \oplus c \geq b \text{ if and only if } c \geq a \to b\]
i.e.\ $a \to b$ is the least element $c$ satisfying $a \oplus c \geq b$.
\end{defi}

First, we give a simple example of a Brouwer algebra.

\begin{defi}\label{defi-upset}
Let $(X,\leq)$ be a poset. We say that a subset $Y \subseteq X$ is \emph{upwards closed} or is an \emph{upset} if for all $y \in Y$ and all $x \in X$ with $x \geq y$ we have $x \in Y$. Similarly, we say that $Y \subseteq X$ is \emph{downwards closed} or a \emph{downset} if for all $y \in Y$ and all $x \in X$ with $x \leq y$ we have $x \in Y$.

We denote by $\oo(X)$ the collection of all upwards closed subsets of $X$, ordered under reverse inclusion $\supseteq$.
\end{defi}

\begin{prop}
$\oo(X)$ is a Brouwer algebra under the operations $U \oplus V = U \cap V$, $U \otimes V = U \cup V$ and
\[U \to V = \{x \in X \mid \forall y \geq x (y \in U \Rightarrow y \in V)\}.\]
\end{prop}
\begin{proof}
The upwards closed sets of a poset form a topology (because they are closed under arbitrary unions and intersections). The result now follows from Balbes and Dwinger \cite[IX.3, Example 4]{balbes-dwinger-1974}.\footnote{In most literature, including Balbes and Dwinger, results are proved for Heyting algebras, the order-dual of Brouwer algebras. However, all results we cite directly follow for Brouwer algebras in the same way.}
\end{proof}

It turns out that the Muchnik lattice is also a Brouwer algebra.

\begin{prop}{\rm(Muchnik \cite{muchnik-1963})}
The Muchnik lattice is a Brouwer algebra under the operations induced by:
\begin{align*}
\aa \oplus \B &= \{f \oplus g \mid f \in \aa \text{ and } g \in \B\}\\
\aa \otimes \B &= \aa \cup \B\\
\aa \to \B &= \{g \in \omega^\omega \mid \forall f \in \A \exists h \in \B(f \oplus g \geq_T h) \}
\end{align*}
\end{prop}

\begin{prop}\label{prop-upw-iso}
The Muchnik lattice is isomorphic to the lattice of upsets of the Turing degrees.
\end{prop}
\begin{proof}
We use a proof inspired by Muchnik's proof that the Muchnik degrees can be embedded in the Medvedev degrees (preserving 0, 1 and minimal upper bounds) from \cite{muchnik-1963}. For every $\A \subseteq \omega^\omega$, we have that $\A \equiv_w C(\A) := \{f \in \omega^\omega \mid \exists g \in \A (g \leq_T f)\}$. Now it is directly verified that the mapping sending $\A$ to $C(\A)$ induces an order isomorphism between $\mw$ and $\oo(\dd)$ (as defined in Definition \ref{defi-upset}). Finally, every order isomorphism between Brouwer algebras is automatically a Brouwer algebra isomorphism, see Balbes and Dwinger \cite[IX.4, Exercise 3]{balbes-dwinger-1974}.
\end{proof}

The main motivation behind Brouwer algebras is that they allow us to specify semantics containing IPC.

\begin{defi}{\rm(McKinsey and Tarski \cite{mckinsey-tarski-1948})}
Let $\phi(x_1,\dots,x_n)$ be a propositional formula with free variables among $x_1,\dots,x_n$, let $B$ be a Brouwer algebra and let $b_1,\dots,b_n \in B$. Let $\psi$ be the formula in the language of Brouwer algebras obtained from $\phi$ by replacing logical disjunction $\vee$ by $\otimes$, logical conjunction $\wedge$ by $\oplus$, logical implication $\to$ by Brouwer implication $\to$ and the false formula $\bot$ by $1$ (we view negation $\neg\alpha$ as $\alpha \to \bot$). We say that $\phi(b_1,\dots,b_n)$ \emph{holds in $B$} if $\psi(b_1,\dots,b_n) = 0$. Furthermore, we define the \emph{theory} of $B$ (notation: $\Th(B)$) to be the set of those formulas which hold for every valuation, i.e.\
\[\Th(B) = \{\phi(x_1,\dots,x_m) \mid \forall b_1,\dots,b_m \in B(\phi(b_1,\dots,b_m) \text{ holds in } B)\}.\]
\end{defi}

The following soundness result is well-known and directly follows from the observation that all rules in some fixed deduction system for IPC preserve truth.

\begin{prop}{\rm(McKinsey and Tarski \cite[Theorem 4.1]{mckinsey-tarski-1948})}
For every Brouwer algebra $B$: $\mathrm{IPC} \subseteq \Th(B)$.
\end{prop}
\begin{proof}
See e.g.\ Chagrov and Zakharyaschev \cite[Theorem 7.10]{chagrov-zakharyaschev-1997}.
\end{proof}

As discussed in the introduction, one might hope that the computationally motivated Muchnik lattice has IPC as its theory. However, it is easily verified that the weak law of the excluded middle $\neg p \vee \neg\neg p$ holds in the Muchnik lattice, while it does not hold in IPC. Fortunately, it turns out we can still capture IPC by looking at certain factors of the Muchnik lattice.

\begin{prop}
Let $B$ be a Brouwer algebra. For every principal filter $\F$ generated by some element $x \in B$, $B / \F$ is a Brouwer algebra (also denoted by $B / x$) under the implication defined on the equivalence classes by
\[[y] \to_{B / \F} [z] = [(y \otimes x) \to_B (z \otimes x)].\]
\end{prop}
\begin{proof}
On one hand we have (because $[y \otimes x] = [y]$ by definition of $B / \F$):
\[[(y \otimes x) \to_B (z \otimes x)] \oplus [y]
= [((y \otimes x) \to_B (z \otimes x)) \oplus (y \otimes x)]
\geq [z \otimes x].\]
On the other hand, for any element $u$ such that $[y] \oplus [u] \geq [z]$ we have that $[y \oplus u \oplus z] = [y \oplus u]$ so $(y \oplus u \oplus z) \otimes x = (y \oplus u) \otimes x$ by definition of $B /x$. Then distributivity shows that
\[(y \otimes x) \oplus (u \otimes x) \oplus (z \otimes x) = (y \otimes x) \oplus (u \otimes x)\]
i.e.\ $(y \otimes x) \oplus (u \otimes x) \geq (z \otimes x)$. So, since $B$ is a Brouwer algebra we see that $u \otimes x \geq (y \otimes x) \to_B (z \otimes x)$, and therefore $[u] = [u \otimes x] \geq [(y \otimes x) \to_B (z \otimes x)]$.
\end{proof}

Taking such a factor essentially amounts to moving from the entire algebra to just the interval $[0,x]_{\mw}$ of elements below $x$ (indeed, the factor is isomorphic to this interval). Because the top element of $[0,x]_{\mw}$ is smaller than the top element of $\mw$ if $x \neq 1$, the interpretation of negation $\neg b$, which is defined as $b \to 1$, also differs between these two algebras. Thus, taking a factor roughly corresponds to changing the negation.

The following result, an analogue of the same result for the Medvedev lattice by Skvortsova \cite{skvortsova-1988}, shows that there exists a factor of the Muchnik lattice with IPC as its theory.

\begin{thm}{\rm(Sorbi and Terwijn \cite{sorbi-terwijn-2012})}
There exists a mass problem $\A \subseteq \omega^\omega$ such that $\Th(\mw / \A) = \mathrm{IPC}$.
\end{thm}

The particular mass problem $\A$ from the previous theorem does not have an intuitive interpretation and is constructed in quite an ad-hoc manner. However, in this paper we will show that natural mass problems $\A$ such that the factor $\mw / \A$ captures IPC do exist.

\section{Splitting classes}\label{sec-splitting}

As announced above, we will present our results in a general framework so that additional examples can easily be obtained. Our framework of \emph{splitting classes} abstracts exactly what we need for our proof in section \ref{sec-ipc} to work. It roughly says that $\aa$ is a splitting class if, given some function $f \in \aa$, we can construct functions $h_0,h_1 \in \aa$ above it whose join is not in $\aa$ while `avoiding' a given finite set of other functions in $\aa$. This is made precise below.

\begin{defi}\label{def-split}
Let $\aa \subseteq \omega^\omega$ be a non-empty countable class which is downwards closed under Turing reducibility. We say that $\aa$ is a \emph{splitting class} if for every $f \in \aa$ and every finite subset $\B \subseteq \{g \in \aa \mid g \not\leq_T f\}$ there exist $h_0,h_1 \in \aa$ such that $h_0,h_1 \geq_T f$, $h_0 \oplus h_1 \not\in \aa$ and for all $g \in \B$: $g \oplus h_0, g \oplus h_1 \not\in \aa$.
\end{defi}

Note that, because every splitting class $\aa$ is downwards closed under Turing reducibility, we in particular have that $\aa$ is closed under Turing equivalence, i.e.\ if $f \in \aa$ and $g \equiv_T f$ then also $g \in \aa$.

We emphasise that we required a splitting class to be countable. There are also interesting examples which satisfy the requirements except for the countability: for example, in section \ref{sec-hif} we will see that this is the case for the set of hyperimmune-free functions. In that section we will also discuss how to suitably generalise the concept to classes of higher cardinality.

It turns our that in order to show that something is a splitting class it will be easier to prove that one of the two alternative formulations given by the next proposition holds.

\begin{prop}\label{prop-split-equiv}
Let $\aa \subseteq \omega^\omega$ be a non-empty countable class which is downwards closed under Turing reducibility. Then the following are equivalent:
\begin{enumerate}[\rm (i)]
\item $\aa$ is a splitting class.
\item For every $f \in \aa$ and every finite subset $\B \subseteq \{g \in \aa \mid g \not\leq_T f\}$ there exists $h \in \aa$ such that $h >_T f$ and for all $g \in \B$: $g \oplus h \not\in \aa$.
\item For every $f \in \aa$ there exists $h \in \aa$ such that $h \not\leq_T f$, and for every $f \in \aa$, every finite subset $\B \subseteq \{g \in \aa \mid g \not\leq_T f\}$ and every $h_0 \in \{g \in \aa \mid g \not\leq_T f\}$ there exists $h_1 \in \aa$ such that $h_1 >_T f$, $h_0 \oplus h_1 \not\in \aa$ and for all $g \in \B$: $h_1 \not\geq_T g$.
\end{enumerate}
\end{prop}
\begin{proof}
(i) $\to$ (ii): Let $h_0,h_1 \in \aa$ be such that $h_0,h_1 \geq_T f$, $h_0 \oplus h_1 \not\in \aa$ and for all $g \in \B$: $g \oplus h_0, g \oplus h_1 \not\in \aa$. Let $h = h_0$. Because $h \equiv_T f$ would imply that $h_0 \oplus h_1 \equiv_T h_1 \in \aa$ we see that $h >_T f$ and therefore we are done.

(ii) $\to$ (iii): First, for every $f \in \aa$ we can find $h \in \aa$ such that $h \not\leq_T f$ by applying (ii) with $\B = \emptyset$. Next, using (ii) determine $h_1 \in \aa$ such that $h_1 >_T f$ and for all $g \in \B \cup \{h_0\}$: $g \oplus h_1 \not\in \aa$. Then the only thing we still need to show is that $h \not\geq_T g$ for all $g \in \B$. However, $h \geq_T g$ would imply $h \oplus g \equiv_T h \in \aa$, a contradiction.

(iii) $\to$ (ii): Fix $g_1 \in \A$ such that $g_1 \not\leq_T f$. Let $\B \subseteq \{g \in \aa \mid g \not\leq_T f\}$ be finite. Without loss of generality, we may assume that $g_1 \in \B$; in particular, we may assume that $\B$ is non-empty. So, let $\B = \{g_1,\dots,g_n\}$. We inductively define a sequence $h_{1,0} <_T h_{1,1} <_T \dots <_T h_{1,n}$ of functions in $\aa$. First, we let $h_{1,0} = f$. Next, to obtain $h_{1,i+1}$ from $h_{1,i}$, apply (iii) to find a function $h_{1,i+1} >_T h_{1,i}$ such that $h_{1,i+1} \oplus g_{i+1} \not\in \aa$ and for all $i+2 \leq j \leq n$ we have $g_j \not\leq_T h_{1,i+1}$. Then $h := h_{1,n}$ is as desired.

(ii) $\to$ (i): Using (ii), we can find $h_0 \in \aa$ such that $h_0 >_T f$ and $g \oplus h_0 \not\in \aa$ for all $g \in \B$. By applying (ii) a second time, we can now find $h_1 \in \aa$ such that $h_1 >_T f$ and for all $g \in \B \cup \{h_0\}$: $g \oplus h_1 \not\in \aa$. Then $h_0$ and $h_1$ are as desired.
\end{proof}

\section{Low and 1-generic below $\emptyset'$ are splitting classes}\label{sec-low}

Before we show that splitting classes allow us to capture IPC as a factor of the Muchnik lattice, we want to demonstrate that our framework of splitting classes is non-trivial. To this end, we will show that the class of low functions, and that the class of functions of 1-generic degree below $\emptyset'$ together with the computable functions, are splitting classes. We will denote the first class by $\A_\mathrm{low}$ and the second class by $\A_{\mathrm{gen}\leq\emptyset'}$. We remark that the second class naturally occurs as the class of functions that are low for EX (as proved in Slaman and Solovay \cite{slaman-solovay-1991}).

Because these kinds of arguments are usually given as constructions on sets (or elements of Cantor space) rather than the functions (or elements of Baire space) which occur in the Muchnik lattice, we will work with sets instead of functions in this section. However, we do not use the compactness of Cantor space anywhere and therefore it is only a notational matter.

First, we recall some basic facts about 1-genericity over a set.

\begin{defi}{\rm(Jockusch \cite[p.\ 125]{jockusch-1980})}
Let $A,B \subseteq \omega$. We say that $B$ is \emph{1-generic} if for every $e \in \omega$ there exists $\sigma \subseteq B$ such that either $\{e\}^{\sigma}(e) \downarrow$ or for all $\tau \supseteq \sigma$ we have $\{e\}^{\tau}(e) \uparrow$.

More generally, we say that $B$ is \emph{1-generic over $A$} if for every $e \in \omega$ there exists $\sigma \subseteq B$ such that either $\{e\}^{A \oplus \sigma}(e) \downarrow$ or for all $\tau \supseteq \sigma$ we have $\{e\}^{A \oplus \tau}(e) \uparrow$.
\end{defi}

\begin{lem}\label{lem-gen-rel}{\rm(Folklore)}
Let $B$ be 1-generic over $A$. Then:
\begin{enumerate}[\rm (i)]
\item If $A$ is 1-generic, then $A \oplus B$ is 1-generic.
\item If $A$ is low and $B \leq_T \emptyset'$, then $A \oplus B$ is low.
\end{enumerate}
\end{lem}
\begin{proof}
(i): Assume $A$ is 1-generic. Let $e \in \omega$ . We need to find a $\sigma \subseteq A \oplus B$ such that either $\{e\}^\sigma(e) \downarrow$ or such that for all $\tau \supseteq \sigma$ we have $\{e\}^\tau(e) \uparrow$.

If $\{e\}^{A \oplus B} (e)\downarrow$, we can choose $\sigma \subseteq A \oplus B$ such that $\{e\}^\sigma(e) \downarrow$. Otherwise, since $B$ is 1-generic over $A$, we can determine $\sigma_B \subseteq B$ such that for all $\tau_B \supseteq \sigma_B$ we have $\{e\}^{A \oplus \tau_B}(e) \uparrow$. Fix an index $\tilde{e}$ such that for all $C \subseteq \omega$ and all $x \in \omega$:
\[\{\tilde{e}\}^C(x) \downarrow \Leftrightarrow \exists \tau_B \supseteq \sigma_B \{e\}^{C \oplus \tau_B}(e) \downarrow.\]

We first note that $\{\tilde{e}\}^{A}(x)\uparrow$ by our choice of $\sigma_B$. Therefore, using the 1-genericity of $A$, determine $\sigma_A \subseteq A$ such that for all $\tau_A \supseteq \sigma_A$ we have $\{\tilde{e}\}^{\tau_A}(\tilde{e}) \uparrow$. By choice of $\tilde{e}$ we then have for for all $\tau_A \supseteq \sigma_A$ that $\forall \tau_B \supseteq \sigma_B \{e\}^{\tau_A \oplus \tau_B}(e) \uparrow$, which is the same as saying that for all $\tau \supseteq \sigma_A \oplus \sigma_B$ we have $\{e\}^{\tau}(e) \uparrow$. This is exactly what we needed to show.

(ii) We show that both $(A \oplus B)'$ and its complement $\overline{(A \oplus B)'}$ are c.e.\ in $A' \oplus B \equiv_T \emptyset'$. To this end, we note that $e \in (A \oplus B)'$ if and only if 
\[\exists \sigma_A \subseteq A \exists \sigma_B \subseteq B\left(\{e\}^{\sigma_A \oplus \sigma_B}(e) \downarrow\right)\]
which is c.e.\ in $A \oplus B \leq_T A' \oplus B$. Next, using the fact that $B$ is 1-generic over $A$, we see that $e \not\in (A \oplus B)'$ if and only if
\[\exists \sigma_B \subseteq B \forall \tau_B \supseteq \sigma_B \left(\{e\}^{A \oplus \tau_B} (e) \uparrow\right)\]
which is c.e.\ in $A' \oplus B$. The result now follows by the relativised Post's theorem.
\end{proof}

\begin{thm}
$\A_\mathrm{low}$ and $\A_{\mathrm{gen}\leq\emptyset'}$ are splitting classes.
\end{thm}
\begin{proof}
The first class is clearly downwards closed; for the second class this is proved in Haught \cite{haught-1986} (but also follows from the fact mentioned above that $\A_{\mathrm{gen}\leq\emptyset'}$ consists of exactly those functions which are low for EX).

First, we consider the class of low functions. By Proposition \ref{prop-split-equiv}, we can show that the low functions form a splitting class by proving that for every low $A$ and every finite $\B \subseteq \{B \in \omega^\omega \mid B\text{ low and } B \not\leq_T A\}$ there exists a set $C \not\leq_T A$ such that $A \oplus C$ is low and such that for all $B \in \B$ we have that $B \oplus (A \oplus C) \equiv_T \emptyset'$. (Note that $C \not \leq_T A$ ensures that $A \oplus C >_T A$, while $B \oplus (A \oplus C) \equiv_T \emptyset'$ ensures that $B \oplus (A \oplus C)$ is neither 1-generic nor low.) Lemma \ref{lem-gen-rel} tells us that we can make $A \oplus C$ low by ensuring that $C \leq \emptyset'$ and that $C$ is 1-generic over $A$. Thus, it is enough if we can show:
\begin{center}
\refstepcounter{equation}(\arabic{equation})\label{low-claim}
\emph{If $A$ is low and $\B \subseteq \{B \in \omega^\omega \mid B \leq_T \emptyset' \text{ and } B \not\leq_T A\}$ is finite, then there exists a set $C \leq_T \emptyset'$ which is 1-generic over $A$ such that $C \not\leq_T A$ and for all $B \in \B$: $B \oplus (A \oplus C) \equiv_T \emptyset'$.}
\end{center}

In fact, we then also immediately get the result for the class of functions of 1-generic degree below $\emptyset'$. 
Namely, let $A \leq_T \emptyset'$ be of 1-generic degree and let $\B  \subseteq \{B \in \omega^\omega \mid B \leq_T \emptyset' \text{ and } B \not\leq_T A\}$ be finite. Just as above, it would be enough to have a set $C \leq_T \emptyset'$ such that $C \not\leq_T A$, $A \oplus C$ is of 1-generic degree and for all $B \in \B$: $B \oplus (A \oplus C) \equiv_T \emptyset'$. Note that this expression is invariant under replacing $A$ with a Turing equivalent set, so because $A$ is of 1-generic degree we may without loss of generality assume $A$ to be 1-generic. Then, because $A \leq \emptyset'$ is 1-generic, it is also low. So we can find a set $C$ as in \eqref{low-claim}. By Lemma \ref{lem-gen-rel} we then have that $A \oplus C$ is 1-generic, and therefore $C$ is exactly as desired.

To prove \eqref{low-claim} we modify the proof of the Posner and Robinson Cupping Theorem \cite{posner-robinson-1981}. Let $\B = \{B_1,\dots,B_k\}$. For every $B_i \in \B$, since $B_i \leq \emptyset'$ we can approximate $B_i$ by a computable sequence $B^0_i,B^1_i,\dots$ of finite sets. We now let $\alpha_i$ be the \emph{computation function} defined by letting $\alpha_i(n)$ be the least $m \geq n$ such that $B^m_i \restriction (n+1) = B_i \restriction (n +1)$. Then $B_i \equiv_T \alpha_i$. Now let $\alpha = \min(\alpha_1,\dots,\alpha_k)$. Then, by Lemma 6 of \cite{posner-robinson-1981}, any function $g$ which dominates $\alpha$ computes some $B_i$. Thus, we see that no function computable in $A$ can dominate $\alpha$.

We will now construct a set $C$ as in \eqref{low-claim} by a finite extension argument, i.e.\ as $C = \bigcup_{n \in \omega} \sigma_n$. Fix any computable sequence $\tau_0,\tau_1,\dots$ of mutually incomparable finite strings (for example, $\tau_n =  \langle 0^n1 \rangle$, the string consisting of $n$ times a $0$ followed by a $1$). We start with $\sigma_0 = \emptyset$. To define $\sigma_{e+1}$ given $\sigma_e$, let $n$ be the least $m \in \omega$ such that either (where the quantifiers are over finite strings):
\begin{equation}\label{pos-rob-1}
\forall \sigma \supseteq \sigma_e \star \tau_m\left(\{e\}^{A \oplus \sigma}(e) \uparrow\right)
\end{equation}
or
\begin{equation}\label{pos-rob-2}
\exists \sigma \supseteq \sigma_e \star \tau_m \left(|\sigma| \leq \alpha(m) \wedge \{e\}^{A \oplus \sigma}(e)[|\sigma|] \downarrow\right).
\end{equation}
Such an $m$ exists: otherwise, for every $l \in \omega$ we could let $\beta(l)$ be the least $s \in \omega$ such that 
\[\exists \sigma \supseteq \sigma_e \star \tau_l\left(\{e\}^{A \oplus \sigma}(e)[|\sigma|] \downarrow \wedge |\sigma| = s\right).\]
For every $l$ such an $s$ exists because \eqref{pos-rob-1} does not hold for $l$, while such an $s$ has to be strictly bigger than $\alpha(l)$ because \eqref{pos-rob-2} also does not hold. So, $\beta$ would be a function computable in $A$ which dominates $\alpha$, of which we have shown above that it cannot exist.

Now, if case \eqref{pos-rob-1} holds for $n$, then we let $\sigma_{e+1} = \sigma_e \star \tau_n \star \emptyset'(e)$. Otherwise, we let $\sigma_{e+1} = \sigma \star \emptyset'(e)$, where $\sigma$ is the least $\sigma$ such that \eqref{pos-rob-2} is satisfied.

The construction is computable in $A' \oplus B_1 \oplus \dots \oplus B_k \leq_T \emptyset'$:  the set of $m \in \omega$ for which \eqref{pos-rob-1} holds is co-c.e.\ in $A$, while for \eqref{pos-rob-2} this is computable in $\alpha \leq_T B_1 \oplus \dots \oplus B_k$ and $A$. Therefore, $C \leq_T \emptyset'$ holds.

Furthermore, per construction of $\sigma_{e+1}$ we have either $\{e\}^{A \oplus \sigma_{e+1}}(e) \downarrow$, or for all $\tau \supseteq \sigma_{e+1}$ we have $\{e\}^{A \oplus \tau}(e) \uparrow$. So, $C$ is 1-generic over $A$.

Next, for every $1 \leq i \leq k$ the construction is computable in $(A \oplus C) \oplus B_i$: to determine $\sigma_{e+1}$ given $\sigma_e$, use $C$ to find the unique $n \in \omega$ such that $C \supseteq \sigma_e \star \tau_n$. We can now compute in $A$ and $B_i$ if there exists some string $\sigma \supseteq \sigma_e \star \tau_n$ of length at most $\alpha_i(n)$ such that $\{e\}^{A \oplus \sigma}(e)[|\sigma|] \downarrow$: if so, let $\sigma$ be the least such string and then $\sigma_{e+1} = B \restriction |\sigma| + 1$. Otherwise, $\sigma_{e+1} = B \restriction |\sigma_e| + 1$. Then we also see that $\emptyset'$ is computable in $(A \oplus C) \oplus B_i$, because $\emptyset'(e)$ is the last element of $\sigma_{e+1}$. Since also $A, B_i, C \leq_T \emptyset'$ we see that $(A \oplus C) \oplus B_i \equiv_T \emptyset'$.

Finally, because for every low $A$ there exists some low $B_0 >_T A$ (see e.g.\ Odifreddi \cite[Proposition V.2.21]{odifreddi-1989}), we may without loss of generality assume that such a $B_0$ is in $\B$. Then we have $B_0 \oplus (A \oplus C) \equiv_T \emptyset'$, as shown above. Now, if it were the case that $C \leq_T A$, then $\emptyset' \equiv_T B_0 \oplus (A \oplus C) \equiv_T B_0$, which contradicts $B_0$ being low. So $C \not\leq_T A$, which is the last thing we needed to show.
\end{proof}

\section{The theory of a splitting class}\label{sec-ipc}

We will now show that the theory of a splitting class equals IPC. We start by moving away from our algebraic viewpoint to Kripke semantics. The crucial step we need for this is the following:

\begin{thm}
For any poset $(X,\leq)$, the theory of $(X,\leq)$ as a Kripke frame is the same as theory of the lattice of upsets of $X$ as a Brouwer algebra.
\end{thm}
\begin{proof}
See e.g.\ Chagrov and Zakharyaschev \cite[Theorem 7.20]{chagrov-zakharyaschev-1997} for the order-dual result for Heyting algebras.
\end{proof}

\begin{prop}\label{prop-much-to-kripke}
Let $\aa \subseteq \omega^\omega$ be downwards closed under Turing reducibility. Then $\mw / \overline{\aa}$ (i.e.\ $\mw$ modulo the principal filter generated by $\overline{\aa}$) is isomorphic to the lattice of upsets $\oo(\aa)$ of $\aa$. In particular, $\Th(\mw / \overline{\aa}) = \Th(\aa)$ (the first as Brouwer algebra, the second as Kripke frame).
\end{prop}
\begin{proof}
By Proposition \ref{prop-upw-iso}, $\mw$ is isomorphic to the lattice of upsets $\oo(\dd)$ of the Turing degrees $\dd$, by sending each set $\B \subseteq \omega^\omega$ to $C(\B)$. Since $\overline{\aa}$ is upwards closed, we see that the isomorphism sends $\overline{\aa}$ to itself. Therefore, $\mw / \overline{\aa}$, which is isomorphic to the initial segment $[\omega^\omega,\overline{\aa}]_\mw$ of $\mw$, is isomorphic to the initial segment $[\omega^\omega,\overline{\aa}]_{\oo(\dd)}$. Finally, $[\omega^\omega,\overline{\aa}]_{\oo(\dd)}$ is easily seen to be isomorphic to $\oo(\aa)$, by sending each set $\B \in \oo(\aa)$ to $\B \cup \overline{\aa}$. The result now follows from the previous theorem.
\end{proof}

Thus, if we take the factor of $\mw$ given by the principal filter generated by $\overline{\aa}$, we get exactly the theory of the Kripke frame $(\aa,\leq_T)$. The rest of this section will be used to show that for splitting classes this theory is exactly IPC. To this end, we need the right kind of morphisms for Kripke frames, called \emph{p-morphisms}.

\begin{defi}{\rm (De Jongh and Troelstra~\cite{dejongh-troelstra-1966})}
Let $(X_1,\leq_1)$, $(X_2,\leq_2)$ be Kripke frames. A surjective function $f: (X_1,\leq_1) \to (X_2,\leq_2)$ is called a \emph{p-morphism} if
\begin{enumerate}
\item $f$ is an order homomorphism: $x \leq_1 y \to f(x) \leq_2 f(y)$,
\item $\forall x \in X_1 \forall y \in X_2 (f(x) \leq_2 y \to \exists z \in X_1 (x \leq_1 z \wedge f(z) = y))$.
\end{enumerate}
\end{defi}

\begin{prop}\label{prop-p-morphism}
If there exists a $p$-morphism from $(X_1,\leq_1)$ to $(X_2,\leq_2)$, then $\Th(X_1,\leq_1) \subseteq \Th(X_2,\leq_2)$.
\end{prop}
\begin{proof}
See e.g.\ Chagrov and Zakharyaschev \cite[Corollary 2.17]{chagrov-zakharyaschev-1997}.
\end{proof}

\begin{thm}\label{thm-tree-ipc}
$\Th(2^{<w}) = \mathrm{IPC}$.
\end{thm}
\begin{proof}
See e.g.\ Chagrov and Zakharyaschev \cite[Corollary 2.33]{chagrov-zakharyaschev-1997}.
\end{proof}

So, if we want to show that the theory of $\mw / \overline{\aa}$ equals IPC, it is enough to show that there exists a $p$-morphism from $\aa$ to $2^{<\omega}$. We next show that this is indeed possible for splitting classes.

\begin{prop}\label{prop-split-tree}
Let $\aa$ be a splitting class. Then there exists a $p$-morphism $\alpha: (\aa,\leq_T) \to 2^{<\omega}$.
\end{prop}
\begin{proof}
Instead of building a $p$-morphism from $\aa$, we will build it from $\aa / {\equiv_T}$ (which is equivalent to building one from $\aa$, since any order homomorphism has to send $T$-equivalence classes to equal strings). For ease of notation we will write $\aa$ for $\aa / \equiv_T$ during the remainder of this proof.

Fix an enumeration $\ba_0,\ba_1,\dots$ of $\aa$. We will build a sequence $\alpha_0 \subseteq \alpha_1 \subseteq \dots$ of finite, partial order homomorphisms from $\aa$ to $2^{<\omega}$, which additionally satisfy that if $\ba,\bb \in \dom(\alpha_i)$ and $\alpha_i(\ba) \mid \alpha_i(\bb)$, then $\ba \oplus \bb \not\in \aa$.

We satisfy the following requirements:
\begin{itemize}
\item $R_0$: $\alpha_0(\bzero) = \emptyset$ (where $\bzero$ is the least Turing degree)
\item $R_{2n+1}$: $\ba_n \in \dom(\alpha_{2n+1})$
\item $R_{2n+2}$: there are $\bc_0,\bc_1 \in \dom(\alpha_{2n+2})$ with $\bc_0,\bc_1 \geq_T \ba_n$ and $\alpha_{2n+2}(\bc_0) = \alpha_{2n+1}(\ba_n) \star 0$, $\alpha_{2n+2}(\bc_1) = \alpha_{2n+1}(\ba_n) \star 1$.
\end{itemize}

First, we show that for such a sequence the function $\alpha = \bigcup_{n \in \omega} \alpha_n$ is a  $p$-morphism $\alpha: (\aa ,\leq_T) \to 2^{<\omega}$. First, the odd requirements ensure that $\alpha$ is total. Furthermore, $\alpha$ is an order homomorphism because the $\alpha_i$ are. To show that $\alpha$ is a $p$-morphism, let $\ba \in \aa $ and let $\alpha(\ba) \subseteq y$; we need to find some $\ba \leq_T \bb \in \aa $ such that $\alpha(\bb) = y$. Because $\alpha(\ba) \subseteq y$ we know that $y = \alpha(\ba) \star y'$ for some finite string $y'$. We may assume $y'$ to be of length 1, the general result then follows by induction. Now, if we let $n \in \omega$ be such that $\ba = \ba_n$ then $\ba_n \in \dom(\alpha_{2n+1})$, so requirement $R_{2n+2}$ tells us that there are functions $\bc_0,\bc_1 \geq \ba_n$ with $\alpha_{2n+2}(\bc_0) = \alpha(\ba) \star 0$ and $\alpha_{2n+2}(\bc_1) = \alpha(\ba) \star 1$. Now either $\alpha(\bc_0) = y$ or $\alpha(\bc_1) = y$, which is what we needed to show. That $\alpha$ is surjective directly follows from the fact that $\emptyset$ is in its range and that it satisfies property (2) of a $p$-morphism.

\bigskip
Now, we show how to actually construct the sequence. First, $\alpha_0$ is already defined. Next assume $\alpha_{2n}$ has been constructed, we will construct $\alpha_{2n+1}$ extending $\alpha_{2n}$ such that $\ba_n \in \dom(\alpha_{2n+1})$. The set
\[X := \{\alpha_{2n}(\bb) \mid \bb \in \dom(\alpha_{2n}) \text{ and } \bb \leq_T \ba_n\}\] is totally ordered under $\subseteq$. Since, if $\bb,\bc \leq_T \ba_n$ then $\bb \oplus \bc \leq_T \ba_n$. Now, if $\alpha_{2n}(\bb)$ and $\alpha_{2n}(\bc)$ are incomparable then we assumed that $\bb \oplus \bc \not\in \aa$. This contradicts the assumption that $\aa$ is downwards closed. So, we can define $\alpha_{2n+1}(\ba_n)$ to be the largest element of $X$.

We show that $\alpha_{2n+1}$ is an order homomorphism; we then also automatically know that it is well-defined. Thus, let $\bb_1,\bb_2 \in \dom(\alpha_{2n+1})$ with $\bb_1 \leq_T \bb_2$. If they are both already in $\dom(\alpha_{2n})$, then the induction hypothesis on $\alpha_{2n}$ already tells us that $\alpha_{2n+1}(\bb_1) \subseteq \alpha_{2n+1}(\bb_2)$. If $\bb_1 \in \dom(\alpha_{2n})$ and $\bb_2 = \ba_n$, then $\alpha_{2n}(\bb_1) \in X$, so by definition of $\alpha_{2n+1}(\ba_n)$ we directly see that $\alpha_{2n+1}(\bb_1) \subseteq \alpha_{2n+1}(\ba_n)$. Finally, we consider the case that $\bb_1 = \ba_n$ and $\bb_2 \in \dom(\alpha_{2n})$. To show that $\alpha_{2n+1}(\ba_n) \subseteq \alpha_{2n+1}(\bb_2) = \alpha_{2n}(\bb_2)$ it is enough to show that all elements of $X$ are below $\alpha_{2n}(\bb_2)$, because $\alpha_{2n+1}(\ba_n)$ is the largest element of the set $X$. Therefore, let $\bb \in \dom(\alpha_{2n})$ be such that $\bb \leq_T \ba_n$. Then we have that $\bb \leq_T \ba_n \leq_T \bb_2$, and since $\alpha_{2n}$ is an order homomorphism this implies that $\alpha_{2n}(\bb) \leq_T \alpha_{2n}(\bb_2)$, as desired.

Finally, we need to show that if $\bc \in \dom(\alpha_{2n})$ is such that $\alpha_{2n+1}(\bc)$ and $\alpha_{2n+1}(\ba_n)$ are incomparable, then $\bc \oplus \ba_n \not\in \aa$. If $\alpha_{2n+1}(\bc)$ and $\alpha_{2n+1}(\ba_n)$ are incomparable, there has to be some $\bb \leq_T \ba_n$ with $\bb \in \dom(\alpha_{2n})$ such that $\alpha_{2n}(\bc)$ and $\alpha_{2n}(\bb)$ are incomparable (because $\alpha_{2n+1}(\ba_n)$ is the largest element of $X$). However, then by induction hypothesis $\bb \oplus \bc \not\in \aa$ and because $\aa$ is downwards closed this also implies that $\bc \oplus \ba_n \not\in \aa$.

\bigskip
We now assume that $\alpha_{2n+1}$ has been defined and consider requirement $R_{2n+2}$. Let $\B = \{\bb \in \dom(\alpha_{2n+1}) \mid \bb \not\leq_T \ba_n\}$. Since $\aa$ is a splitting class there exist $\bc_0,\bc_1 \in \aa$ such that $\bc_0,\bc_1 \geq \ba_n$, $\bc_0 \oplus \bc_1 \not\in \aa$ and for all $\bb \in \B$ we have $\bb \oplus \bc_0, \bb \oplus \bc_1 \not\in \aa$. Now extend $\alpha_{2n+1}$ by letting $\alpha_{2n+2}(\bc_0) = \alpha_{2n+1}(\ba_n) \star 0$ and $\alpha_{2n+2}(\bc_1) = \alpha_{2n+1}(\ba_n) \star 1$.

First, we show that $\alpha_{2n+2}$ is an order homomorphism. Let $\bb_1,\bb_2 \in \dom(\alpha_{2n+2})$ and $\bb_1 \leq_T \bb_2$. We again distinguish several cases:
\begin{itemize}
\item $\bb_1,\bb_2 \in \dom(\alpha_{2n+1})$: this directly follows from the fact that $\alpha_{2n+1}$ is an order homomorphism by induction hypothesis.
\item $\bb_1,\bb_2 \in \{\bc_0,\bc_1\}$: since $\bc_0 \oplus \bc_1 \not\in \aa$ and therefore differs from both $\bc_0$ and $\bc_1$, this can only happen if $\bb_1 = \bb_2$, so this case is trivial.
\item $\bb_1 \in \{\bc_0,\bc_1\}$, $\bb_2 \in \dom(\alpha_{2n+1})$: note that $\bc_0,\bc_1 >_T \ba_n$ (otherwise $\bc_0 \oplus \bc_1 \in \{\bc_0,\bc_1\} \subseteq \aa$), so we see that $\bb_2 >_T \ba$, and then by construction of $\bc_0$ and $\bc_1$ we know that $\bb_2 \oplus \bc_0, \bb_2 \oplus \bc_1 \not\in \aa$. This contradicts $\bb_1 \leq_T \bb_2$, so this case is impossible.
\item $\bb_1 \in \dom(\alpha_{2n+1})$, $\bb_2 \in \{\bc_0,\bc_1\}$: if $\bb_1 \not\leq_T \ba_n$, then again by construction of $\bc_0$ and $\bc_1$ we have that $\bb_2 = \bb_1 \oplus \bb_2 \not\in \aa$ which is a contradiction. So $\bb_1 \leq_T \ba_n$ and therefore $\alpha_{2n+2}(\bb_1) = \alpha_{2n+1}(\bb_1) \subseteq \alpha_{2n+1}(\ba_n) \subseteq \alpha_{2n+2}(\bb_2)$. 
\end{itemize}

Finally, we show that if $\bb \in \dom(\alpha_{2n+2})$ is such that $\alpha_{2n+2}(\bb)$ and $\alpha_{2n+2}(\bc_1)$ are incomparable, then $\bb \oplus \bc_1 \not\in \aa$ (the same then follows analogously for $\bc_2$). If $\bb = \bc_2$ this is clear from the definition of $\alpha_{2n+2}$. Otherwise, we have $\bb \in \dom(\alpha_{2n+1})$.
If it were the case that $\bb \leq_T \ba_n$, then $\alpha_{2n+2}(\bb) = \alpha_{2n+1}(\bb) \subseteq \alpha_{2n+1}(\ba_n) \subseteq \alpha_{2n+2}(\bc_1)$, a contradiction. Thus $\bb \not\leq_T \ba_n$, and therefore $\bb \oplus \bc_1 \not\in \aa$ by construction of $\bc_1$.
\end{proof}

\begin{thm}
For any splitting class $\aa$: $\Th(\mw / \overline{\aa}) = \mathrm{IPC}$.
\end{thm}
\begin{proof}
From Proposition \ref{prop-much-to-kripke}, Proposition \ref{prop-p-morphism}, Theorem \ref{thm-tree-ipc} and Proposition \ref{prop-split-tree}.
\end{proof}

Therefore, combining this with the results from section \ref{sec-low} we now see:
\begin{thm}
$\Th(\mw / \overline{\A_\mathrm{low}}) = \Th(\mw / \overline{\A_\mathrm{\mathrm{gen}\leq\emptyset'}}) = \mathrm{IPC}.$
\end{thm}

\section{Further splitting classes}\label{sec-hif}

\subsection{Hyperimmune-free functions}
In this section, we will look at some other classes and consider if they are splitting classes.
First, we look at the class of hyperimmune-free functions. Recall that a function $f$ is \emph{hyperimmune-free} if every $g \leq_T f$ is dominated by a computable function.  We can see a problem right away: the class of hyperimmune-free functions is well-known to be uncountable, while we required splitting classes to be countable. We temporarily remedy this by only looking at the hyperimmune-free functions which are low$_2$ (where a function $f$ is low$_2$ if $f'' \equiv_T \emptyset''$); after the proof, we will discuss how we might be able to look at the entire class.\footnote{There are different natural countable subsets of the hyperimmune-free degrees which form splitting classes; for example, instead of the low$_2$ hyperimmune-free functions we could also take the hyperimmune-free functions $f$ for which there exists an $n \in \omega$ such that $f \leq_T \emptyset^{(n)}$. This follows from the proof of Theorem \ref{thm-hif-split}. However, since our main reason to look at these countable subclasses is to view them as a stepping stone towards the class of all hyperimmune-free functions, we will not pursue this topic further.}

As in section \ref{sec-low} we will present our constructions as constructions on Cantor space rather than Baire space for the reasons discussed in that section.

\begin{thm}\label{thm-hif-split}
The class $\A_{\mathrm{HIF,low}_2}$ of hyperimmune-free functions which are low$_2$ is a splitting class. In particular, $\Th(\mw / \overline{\A_{\mathrm{HIF,low}_2}}) = \mathrm{IPC}$.
\end{thm}
\begin{proof}
We prove that (iii) of Proposition \ref{prop-split-equiv} holds. That for every hyperimmune-free low$_2$ set $A$ there exists a hyperimmune-free low$_2$ set $B$ such that $B \not\leq_T A$ (or that there even exists one such that $B >_T A$) is well-known, see Miller and Martin \cite[Theorem 2.1]{miller-martin-1968}. We prove the second part of (iii) from Proposition \ref{prop-split-equiv}. Our construction uses the tree method of Miller and Martin \cite{miller-martin-1968}.

Let $A \leq_T \emptyset''$ be hyperimmune-free and low$_2$, let
\[\B \subseteq \{B \subseteq \omega \mid B \not\leq_T A, B \leq_T \emptyset'' \text{ and } B \text{ HIF}\}\]
be a finite subset and let $C_0 \leq_T \emptyset''$ be a hyperimmune-free (low$_2$) set not below $A$. We need to construct a hyperimmune-free set $A <_T C_1 \leq_T \emptyset''$ such that $C_0 \oplus C_1$ is not of hyperimmune-free degree (i.e.\ of hyperimmune degree) and such that for all $B \in \B$ we have that $C_1 \not\geq_T B$.

First, we remark that we may assume that not only $C_0 \not\leq_T A$, but even that $C_0 \not\leq_T A'$. Indeed, assume $C_0 \leq_T A'$. If $C_0 \geq_T A$ then we see that $A < C_0 \leq A'$ so by Miller and Martin \cite[Theorem 1.2]{miller-martin-1968} we see that $C_0$ is  of hyperimmune degree, contrary to our assumption. So, $C_0 \mid_T A$. However, then $A <_T A \oplus C_0 \leq A'$ and as before we then see that $A \oplus C_0$ is already of hyperimmune degree, so we may take $C_1$ to be any hyperimmune-free set strictly above $A$ which is low$_2$ (such a set can be directly constructed using the construction of Miller and Martin).

Without loss of generality we may even assume that $C_0$ is not c.e.\ in $A'$: we may replace $C_0$ by $C_0 \oplus \overline{C_0}$, which is of the same Turing degree as $C_0$, and is not c.e.\ in $A'$ because otherwise $C_0$ would be computable in $A'$, a contradiction.

Let $\B = \{B_1,\dots,B_n\}$ and fix a computable enumeration $\alpha$ of $n \times \omega$. We will construct a sequence $T_0 \supseteq T_1 \supseteq \dots$ of $A$-computable binary trees (in the sense of Odifreddi \cite[Definition V.5.1]{odifreddi-1989}) such that:
\begin{enumerate}[\rm (i)]
\item $T_0$ is the full binary tree.
\item For all $D$ on $T_{4e+1}$: $D \not= \{e\}^A$.
\item For $T_{4e+2}$, one of the following holds:
\begin{enumerate}
\item For all $D$ on $T_{4e+2}$, $\{e\}^{A \oplus D}$ is not total.
\item For all $D$ on $T_{4e+2}$, $\{e\}^{A \oplus D}$ is total and 
\[\forall n \forall \sigma (|\sigma| = n \rightarrow \{e\}^{A \oplus T_{4e+2}(\sigma)}(n)[|T_{4e+2}(\sigma)|] \downarrow).\]
\end{enumerate}
Furthermore, this choice is computable in $\emptyset''$.
\item For all $D$ on $T_{4e+3}$, $\{\alpha_2(e)\}^{A \oplus D} \not= B_{\alpha_1(e)}$.
\item $T_{4e+4}$ is the full subtree of $T_{4e+3}$ above $T_{4e+3}(\langle\emptyset''(e)\rangle)$.
\item For every infinite branch $D$ on all of the trees $T_i$, the sequence $T_0 \supseteq T_1 \supseteq \dots$ is computable in $C_0 \oplus (A' \oplus D)$.
\item The sequence $T_0 \supseteq T_1 \supseteq \dots$ is computable in $\emptyset''$.
\end{enumerate}

For now, assume we can construct such a sequence. Let $D = \bigcup_{i \in \omega}T_i(\emptyset)$, then $D$ is an infinite branch lying on all of the $T_i$. Let $C_1 = A \oplus D$. Then the requirements (ii) guarantee that $D \not\leq_T A$ and therefore $C_1 >_T A$. By (vii) we also have that $C_1 \leq_T \emptyset''$. Furthermore, the requirements (iii) enforce that $C_1$ is hyperimmune-free relative to $A$ (due to Miller and Martin, see e.g.\ Odifreddi \cite[Proposition V.5.6]{odifreddi-1999}), and because $A$ is itself hyperimmune-free it is directly seen that $C_1$ is hyperimmune-free. The requirements (iv) ensure that $C_1 \not\geq_T B_i$ for all $B_i \in \B$.

Next, we have that $(C_0 \oplus C_1)' \geq_T C_0 \oplus (A' \oplus D) \geq_T \emptyset''$: by requirement (vi) the sequence $T_i$ is computable in $C_0 \oplus (A' \oplus D)$, while by requirement (v) we have that $T_{4e+4}(\emptyset) = T_{4e+3}(\emptyset) \star \emptyset''(e)$ which allows us to recover $\emptyset''(e)$. So, $C_0 \oplus C_1$ is not low$_2$. In fact, $C_0 \oplus C_1$ is not even hyperimmune-free: by a theorem of Martin $(C_0 \oplus C_1)' \geq_T \emptyset''$ implies that $C_0 \oplus C_1$ computes a function which dominates every total computable function (see e.g.\ Odifreddi \cite[Theorem XI.1.2]{odifreddi-1999}), and therefore $C_0 \oplus C_1$ is not hyperimmune-free, as desired.

Finally, we show that $C_1$ is low$_2$. By requirement (iv) and requirement (vii) we have that $\emptyset'' \geq_T \{e \in \omega \mid \{e\}^{C_1} \textrm{ is total}\}$. Since the latter has the same Turing degree as $C_1''$, this shows that $C_1$ is indeed low$_2$. 

\bigskip
We now show how to actually construct such a sequence of computable binary trees. Let $T_0$ be the full binary tree. Next, assume $T_{4e}$ has already been defined. To fulfil requirement (ii), observe that  $T_{4e}(0)$ and $T_{4e}(1)$ are incompatible, so at least one of them has to differ from $\{e\}^A$. If the first differs from $\{e\}^A$ we take $T_{4e+1}$ to be the full subtree above $T_{4e}(0)$, and otherwise we take the full subtree above $T_{4e}(1)$.

Next, assume $T_{4e+1}$ has been defined, we will construct $T_{4e+2}$ fulfilling requirement (iii). Let $n$ be the smallest $m \in \omega$ such that either
\begin{equation}\label{hif-1}
m \not\in C_0 \wedge \exists \sigma \supseteq \langle 0^m1\rangle \exists x \forall \tau \supseteq \sigma \left(\{e\}^{A \oplus T_{4e+1}(\tau)}(x) \uparrow\right)
\end{equation}
or
\begin{equation}\label{hif-2}
m \in C_0 \wedge \forall \sigma \supseteq \langle 0^m1\rangle \forall x \exists \tau \supseteq \sigma \left(\{e\}^{A \oplus T_{4e+1}(\tau)}(x) \downarrow\right),
\end{equation}
where as before $\langle 0^m1\rangle$ denotes the string consisting of $m$ times a $0$ followed by a $1$.

Such an $m$ exists: indeed, if such an $m$ did not exist, then
\[C_0 = \left\{m \in \omega \mid \exists \sigma \supseteq \langle 0^m1\rangle \exists x \forall \tau \supseteq \sigma \left(\{e\}^{A \oplus T_{4e+1}(\tau)}(x) \uparrow\right)\right\}\]
and therefore $C_0$ is c.e.\ in $A'$, which contradicts our assumption above.

If \eqref{hif-1} holds for $n$, let $\sigma \supseteq \langle 0^n1\rangle$ be the smallest such string and let $T_{4e+2}$ be the full subtree above $T_{4e+1}(\sigma)$. Otherwise, we inductively define $T_{4e+2} \subseteq T_{4e+1}$. First, if we let $\tau$ be the least $\tilde{\tau} \supseteq \langle 0^n1\rangle$ such that $\{e\}^{A \oplus T_{4e+1}(\tilde{\tau})}(0)[|T_{4e+1}(\tilde{\tau})|] \downarrow$, then we let $T_{4e+2}(0) = T_{4e+1}(\tau)$. Inductively, given $T_{4e+2}(\sigma)$, let $\rho$ be such that $T_{4e+2}(\sigma) = T_{4e+1}(\rho)$. Now, if we let $\tau$ be the least $\tilde{\tau} \supseteq \rho$ such that $\{e\}^{A \oplus T_{4e+1}(\tilde{\tau})}(|\sigma|+1)[|T_{4e+1}(\tilde{\tau})|] \downarrow$, we let $T_{4e+2}(\sigma \star 0) = T_{4e+1}(\tau \star 0)$ and $T_{4e+2}(\sigma \star 1) = T_{4e+1}(\tau \star 1)$.

For the requirements (iv) we do something similar. Let $\tilde{e} = \alpha_2(e)$. First, we build a subtree $S \subseteq T_{4e+2}$ such that either there is no $\tilde{e}$-splitting relative to $A$ on $S$ (i.e.\ for all strings $\sigma,\tau$ on $S$ and all $x \in \omega$, if $\{\tilde{e}\}^{A \oplus \sigma}(x) \downarrow$ and $\{\tilde{e}\}^{A \oplus \tau}(x) \downarrow$, then their values are equal), or $S(0)$ and $S(1)$ are an $\tilde{e}$-splitting relative to $A$ (in fact, $S$ will even be an $\tilde{e}$-splitting tree relative to $A$). Let $n$ be the smallest $m \in \omega$ such that
\begin{align}
m \not\in C_0 \wedge \exists \sigma \supseteq \langle 0^m1\rangle \forall \tau,\tau' \supseteq \sigma \forall x \Big(\{\tilde{e}\}^{A \oplus T_{4e+2}(\tau)}(x) \downarrow \wedge \{\tilde{e}\}^{A \oplus T_{4e+2}(\tau')}(x) \downarrow\notag\\
 \to \{\tilde{e}\}^{A \oplus T_{4e+2}(\tau)}(x)  = \{\tilde{e}\}^{A \oplus T_{4e+2}(\tau')}(x)\Big)\label{hif-3}
\end{align}
or
\begin{align}
m \in C_0 \wedge \forall \sigma \supseteq \langle 0^m1\rangle \exists \tau,\tau' \supseteq \sigma \exists x \Big(\{\tilde{e}\}^{A \oplus T_{4e+2}(\tau)}(x) \downarrow \wedge \{\tilde{e}\}^{A \oplus T_{4e+2}(\tau')}(x) \downarrow\notag\\
 \wedge\;\{\tilde{e}\}^{A \oplus T_{4e+2}(\tau)}(x)  \not= \{\tilde{e}\}^{A \oplus T_{4e+2}(\tau')}(x)\Big)\label{hif-4}
\end{align}

That such an $m$ exists can be shown in the same way as above. If \eqref{hif-3} holds for $n$, let $\sigma$ be the smallest such string and let $S$ be the full subtree above $T_{4e+2}(\sigma)$. Then there are no $\tilde{e}$-splittings relative to $A$ on $S$. Otherwise, we can inductively build $S$: let $S(\emptyset) = T_{4e+2}(\langle 0^n1\rangle)$ and if $S(\sigma)$ is already defined we can take $S(\sigma \star 0)$ and $S(\sigma \star 1)$ to be two $\tilde{e}$-splitting extensions relative to $A$ of $S(\sigma)$ on $T_{4e+2}$.

If there are no $\tilde{e}$-splittings relative to $A$ on $S$, then we can take $T_{4e+3} = S$. Since, assume $\{\tilde{e}\}^{A \oplus D} = B_i$ for some $B_i \in \B$. Then, by Spector's result (see e.g. Odifreddi \cite[Proposition V.5.9]{odifreddi-1999}) we have that $B_i \leq_T A$, contrary to assumption.

Otherwise we can find an $x \in \omega$ such that $\{\tilde{e}\}^{A \oplus S(\langle0\rangle)}(x)$ and $\{\tilde{e}\}^{A \oplus S(\langle1\rangle)}(x)$ both converge, but such that their value differs. Then either $\{\tilde{e}\}^{A \oplus S(\langle0\rangle)}(x) \neq B_{\alpha_1(e)}$ and we take $T_{4e+3}$ to be the full subtree above $S(\langle 0 \rangle)$, or $\{\tilde{e}\}^{A \oplus S(\langle1\rangle)}(x) \neq B_{\alpha_1(e)}$ and we take $T_{4e+3}$ to be the full subtree above $S(\langle 1 \rangle)$. Then $T_{4e+3}$ satisfies requirement (iv).

Finally, how to define $T_{4e+4}$ from $T_{4e+3}$ is already completely specified by requirement (v). This completes the definitions of all the $T_i$. Note that all steps in the construction are computable in $A'' \equiv_T \emptyset''$.

So, the last thing we need to show is that requirement (vi) is satisfied, i.e.\ that for any infinite branch $D$ on all $T_i$ the construction is computable in $C_0 \oplus (A' \oplus D)$. This is clear for the construction of $T_{4e+1}$ from $T_{4e}$. For the construction of $T_{4e+2}$ from $T_{4e+1}$ the only real problem is that we need to choose between \eqref{hif-1} and \eqref{hif-2}. However, because $D$ is on $T_{4e+2}$, we can uniquely determine $n \in \omega$ such that $T_{4e+1}(\langle 0^n1\rangle)$ is an initial segment of $D$. Then \eqref{hif-1} holds if and only if $n \not\in C_0$ and \eqref{hif-2} holds if and only if $n \in C_0$.  So, we can decide which alternative was taken using $C_0$. Furthermore, if \eqref{hif-1} holds then we can use $A'$ to calculate the string $\sigma$ used in the computation of $T_{4e+2}$.

For $T_{4e+3}$ we can do something similar for the tree $S$ used in the definition of $T_{4e+3}$, and using $D$ we can determine if we took $T_{4e+3}$ to be the subtree above $S(\langle 0 \rangle)$ or $S(\langle 1 \rangle)$. Finally, using $D$ it is also easily decided which alternative we took for $T_{4e+4}$, because $T_{4e+4}$ is the full subtree above $T_{4e+3}(\langle i \rangle)$ for the unique $i \in \{0,1\}$ such that $T_{4e+3}(\langle i \rangle) \subseteq D$. Therefore we see that the construction is indeed computable in $C_0 \oplus (A' \oplus D)$, which completes our proof.
\end{proof}

This result is slightly unsatisfactory because we restricted ourselves to the hyper\-immune-free which are low$_2$. Because the entire class of hyperimmune-free functions $\A_\mathrm{HIF}$ is also downwards closed we directly see from the proof above that the only real problem is the uncountability, i.e.\ $\A_\mathrm{HIF}$ satisfies all properties of a splitting class except for the countability. Our next result shows that, if we assume the continuum hypothesis, we can still show that the theory of the factor given by $\A_\mathrm{HIF}$ is IPC.

\begin{defi}
Let $\aa \subseteq \omega^\omega$ be a non-empty class of cardinality $\aleph_1$ which is downwards closed under Turing reducibility. We say that $\aa$ is an $\aleph_1$ \emph{splitting class} if for every $f \in \aa$ and every countable subset $\B \subseteq \{g \in \aa \mid g \not\leq_T f\}$ there exist $h_0,h_1 \in \aa$ such that $h_0,h_1 \geq_T f$, $h_0 \oplus h_1 \not\in \aa$ and for all $g \in \B$: $g \oplus h_0, g \oplus h_1 \not\in \aa$.
\end{defi}

\begin{prop}\label{prop-split-equiv-2}
Let $\aa \subseteq \omega^\omega$ be a non-empty class of cardinality $\aleph_1$ which is downwards closed under Turing reducibility. Then the following are equivalent:
\begin{enumerate}[\rm (i)]
\item $\aa$ is an $\aleph_1$ splitting class.
\item For every $f \in \aa$ and every countable subset $\B \subseteq \{g \in \aa \mid g \not\leq_T f\}$ there exists $h \in \aa$ such that $h >_T f$ and for all $g \in \B$: $g \oplus h \not\in \aa$.
\end{enumerate}
Furthermore, if every countable chain in $\aa$ has an upper bound in $\aa$, these two are also equivalent to:
\begin{enumerate}[\rm (i)]
\setcounter{enumi}{2}
\item For every $f \in \aa$ there exists $h \in \aa$ such that $h \not\leq_T f$, and for every $f \in \aa$, every countable subset $\B \subseteq \{g \in \aa \mid g \not\leq_T f\}$ and every $h_0 \in \{g \in \aa \mid g \not\leq_T f\}$ there exists $h_1 \in \aa$ such that $h_1 >_T f$, $h_0 \oplus h_1 \not\in \aa$ and for all $g \in \B$: $h_1 \not\geq_T g$.
\end{enumerate}
\end{prop}
\begin{proof}
In almost exactly the same way as Proposition \ref{prop-split-equiv}. For the implication (iii) $\to$ (ii) we define an infinite sequence $h_{1,0} <_T h_{1,1} <_T \dots$ instead of a finite one, and then let $h$ be an upper bound in $\aa$ of this chain.
\end{proof}

\begin{thm}
For any $\aleph_1$ splitting class $\aa$: $\Th(\mw / \overline{\aa}) = \mathrm{IPC}$.
\end{thm}
\begin{proof}
We can generalise the construction in Proposition \ref{prop-split-tree} to a transfinite construction over $\aleph_1$. However, instead of building a $p$-morphism to $2^{< \omega}$ we show that we can build a $p$-morphism to every finite binary tree of the form $2^{< n}$. This is already enough to show that the theory is IPC (see e.g.\ Chagrov and Zakharyaschev \cite[Corollary 2.33]{chagrov-zakharyaschev-1997}).

Fix an enumeration $(\ba_\gamma)_{\gamma < \aleph_1}$ of $\aa$. This time we will build a sequence $(\alpha_\gamma)_{\gamma < \aleph_1}$ of partial order homomorphisms from $\aa$ to $2^{<n}$ with countable domain, which is increasing in the sense that $\alpha_\gamma \subseteq \alpha_{\tilde{\gamma}}$ if $\gamma \leq \tilde{\gamma}$. As before, it should additionally satisfy that if $\ba,\bb \in \dom(\alpha_\gamma)$ and $\alpha_\gamma(\ba) \mid \alpha_\gamma(\bb)$, then $\ba \oplus \bb \not\in \aa$.

Fix some bijection $\zeta:  \{0,1\} \times \aleph_1 \to \aleph_1 \setminus \{0\}$ satisfying that $\zeta(1,\gamma) > \zeta(0,\gamma)$ for every $\gamma < \aleph_1$. We satisfy the requirements:
\begin{itemize}
\item $R_0$: $\alpha_0(\bzero) = \emptyset$
\item $R_{(0,\gamma)}$: $\ba_\gamma \in \dom(\alpha_{\zeta(0,\gamma)})$
\item $R_{(1,\gamma)}$: if $\sigma :=\alpha_{\zeta(0,\gamma)}(\ba_\gamma)$ is not maximal in $2^{<n}$, then there are $\bc_0,\bc_1 \in \dom(\alpha_{\zeta(1,\gamma)})$ with $\bc_0,\bc_1 \geq_T \ba_\gamma$ and $\alpha_{\zeta(1,\gamma)}(\bc_0) = \sigma \star 0$, $\alpha_{\zeta(1,\gamma)}(\bc_1) = \sigma \star 1$.
\end{itemize}

That these requirements give us the required $p$-morphisms follows in the same way as in Proposition \ref{prop-split-tree}. The construction of the sequence $(\alpha_\gamma)_{\gamma < \aleph_1}$ also proceeds in almost the same way, apart from two minor details. First, if $\gamma$ is a limit ordinal it does not have a clear predecessor, so we cannot say that $\alpha_\gamma$ should extend its predecessor. Instead, we construct $\alpha_\gamma$ as an extension of $\bigcup_{\tilde{\gamma} < \gamma} \alpha_{\tilde{\gamma}}$ (note that this union is countable because $\gamma < \aleph_1$, and hence $\bigcup_{\tilde{\gamma} < \gamma} \alpha_{\tilde{\gamma}}$ has countable domain). Secondly, the domains of the $\alpha_\gamma$ are no longer finite but are now countable, which means that in the construction for requirement $R_{(1,\gamma)}$ we now need to consider countable sets $\B$ instead of just finite sets $\B$. However, this is exactly why we changed our definition of an $\aleph_1$ splitting class to allow countable sets $\B$ instead of just finite sets $\B$.
\end{proof}

\begin{thm}\label{thm-hif-split-ch}
Assume $\mathrm{CH}$. Then $\A_\mathrm{HIF}$ is an $\aleph_1$ splitting class. In particular, $\Th(\mw / \overline{\A_\mathrm{HIF}}) = \mathrm{IPC}$.
\end{thm}
\begin{proof}
First, $\A_\mathrm{HIF}$ has cardinality $\aleph_1$ by CH.
Next, every countable chain in $\A_\mathrm{HIF}$ has an upper bound in $\A_\mathrm{HIF}$ (Miller and Martin \cite[Theorem 2.2]{miller-martin-1968}), so we can use the equivalence of (i) and (iii) of Proposition \ref{prop-split-equiv-2}. Thus, it is sufficient if we show that the construction in Theorem \ref{thm-hif-split} not only applies to just finite sets $\B$, but also to countable sets $\B$. However, this is readily verified.
\end{proof}

In particular, we see that it is consistent (relative to ZFC) to have $\Th(\mw / \overline{\A_\mathrm{HIF}}) = \mathrm{IPC}$. Unfortunately, we currently do not know if this already follows from ZFC or if it is independent of ZFC.

\begin{question}
Does $\Th(\mw / \overline{\A_\mathrm{HIF}}) = \mathrm{IPC}$ follow from $\mathrm{ZFC}$?
\end{question}

\subsection{Computably traceable functions}

A class that is closely related to the hyperimmune-free functions is the class $\A_\mathrm{trace}$ of computably traceable functions. We first recall its definition.

\begin{defi}{\rm(Terwijn and Zambella \cite{terwijn-zambella-2001})}
A set $T \subseteq \omega \times \omega$ is called a \emph{trace} if all sections $T^{[k]} = \{n \in \omega \mid (k,n) \in T\}$ are finite. A \emph{computable trace} is a trace such that the function which maps $k$ to the canonical index of $T^{[k]}$ is computable. A trace $T$ \emph{traces} a function $g$ if $g(k) \in T^{[k]}$ for every $k \in \omega$. A \emph{bound} is a function $h: \omega \to \omega$ that is non-decreasing and has infinite range. If $|T^{[k]}| \leq h(k)$ for all $k \in \omega$, we say that $h$ is a \emph{bound for} $T$.

Finally, a function $f$ is called \emph{computably traceable} if there exists a computable bound $h$ such that all (total) functions $g \leq_T f$ are traced by a computable trace bounded by $h$.
\end{defi}

Computable traceability can be seen as a uniform kind of hyperimmune-freeness. If $f$ is computably traceable, then it is certainly hyperimmune-free: if $g \leq_T f$ is traced by some computable trace $T$, then for the computable function $\tilde{g}(k) = \max\left(T^{[k]}\right)$ we have $g \leq \tilde{g}$. Conversely, if $f$ is hyperimmune-free and $g \leq_T f$, then $g$ has a computable trace: fix some computable $\tilde{g} \geq g$ and let $T_g = \{(k,m) \mid m \leq \tilde{g}(k)\}$. However, these traces $T_g$ need not be bounded by any uniform computable bound $h$. Computable traceability asserts that such a uniform bound does exist. It can be shown that there are hyperimmune-free functions which are not computably traceable, see Terwijn and Zambella \cite{terwijn-zambella-2001}.

The computably traceable functions naturally occur in algorithmic randomness. In \cite{terwijn-zambella-2001} it is shown that the computably traceable functions are precisely those functions which are low for Schnorr null, and in Kjos-Hanssen, Nies and Stephan \cite{kjoshanssen-nies-stephan-2005} it is shown that this class also coincides with the functions which are low for Schnorr randomness.

Terwijn and Zambella also showed that the usual Miller and Martin tree construction of hyperimmune-free degrees actually already yields a computably traceable degree. Combining their techniques with the next lemma, we can directly see that our constructions of hyperimmune-free degrees above can also be used to construct computably traceable degrees.

\begin{lem}
Let $A$ be computably traceable, and let $B$ be computably traceable relative to $A$. Then $B$ is computably traceable.
\end{lem}
\begin{proof}
Let $h_1$ be a computable bound for the traces of functions computed by $A$ and let $h_2 \leq_T A$ be a bound for the traces of functions computed by $B$. Because $A$ is hyperimmune-free (as discussed above) $h_2$ is bounded by a computable function $\tilde{h}_2$.We claim: every function computed by $B$ has a trace bounded by the computable function $h_1 \cdot \tilde{h}_2$.

To this end, let $g \leq_T B$. Fix a trace $T \leq_T A$ for $g$ which is bounded by $h_2$ (and hence is also bounded by $\tilde{h}_2$). Then the function mapping $k$ to the canonical index of $T^{[k]}$ is computable in $A$, so because $A$ is computably traceable we can determine a computable trace $S$ for this function which is bounded by $h_1$.

Finally, denote by $D_{e,n}$ the (at most) $n$ smallest elements of the set $D_e$ corresponding to the canonical index $e$; i.e.\ $D_{e,n}$ consists of the $n$ smallest elements of $D_e$ if $|D_e| \geq n$, and $D_{e,n} = D_e$ otherwise. Now let $U$ be the computable trace such that $U^{[k]} = \bigcup_{e \in S^{[k]}} D_{e,\tilde{h}_2(k)}$. Then $U$ is clearly bounded by $h_1 \cdot \tilde{h}_2$. It also traces $g$, because $g(k) \in T^{[k]}$ and for some $e \in S^{[k]}$ we have $T^{[k]} = D_{e,\tilde{h}_2(k)}$.
\end{proof}

\begin{thm}
The class $\A_{\mathrm{trace,low}_2}$ of computably traceable functions which are low$_2$  is a splitting class. In particular, $\Th(\mw / \overline{\A_{\mathrm{trace,low}_2}}) = \mathrm{IPC}$.
\end{thm}
\begin{proof}
As in Theorem \ref{thm-hif-split}.
\end{proof}

\begin{thm}
Assume $\mathrm{CH}$. Then $\A_\mathrm{trace}$ is an $\aleph_1$ splitting class. In particular, $\Th(\mw / \overline{\A_\mathrm{trace}}) = \mathrm{IPC}$.
\end{thm}
\begin{proof}
As in Theorem \ref{thm-hif-split-ch}.
\end{proof}

\begin{question}
Does $\Th(\mw / \overline{\A_\mathrm{trace}}) = \mathrm{IPC}$ follow from $\mathrm{ZFC}$?
\end{question}

\subsection*{Acknowledgements}

The author thanks Sebastiaan Terwijn for helpful discussions on the subject and for suggesting some of the classes studied above. Furthermore, the author thanks the anonymous referee for his help in correcting an error in section \ref{sec-hif}.

\providecommand{\bysame}{\leavevmode\hbox to3em{\hrulefill}\thinspace}
\providecommand{\MR}{\relax\ifhmode\unskip\space\fi MR }
\providecommand{\MRhref}[2]{%
  \href{http://www.ams.org/mathscinet-getitem?mr=#1}{#2}
}
\providecommand{\href}[2]{#2}


\begin{thebibliography}{10}

\bibitem{balbes-dwinger-1974}
R.~Balbes and P.~Dwinger, \emph{Distributive lattices}, University of Missouri
  Press, 1974.

\bibitem{chagrov-zakharyaschev-1997}
A.~Chagrov and M.~Zakharyaschev, \emph{Modal logic}, Clarendon Press, 1997.

\bibitem{dejongh-troelstra-1966}
D.~H.~J. de~Jongh and A.~S. Troelstra, \emph{On the connection of partially
  ordered sets with some pseudo-{B}oolean algebras}, Indagationes Mathematicae
  \textbf{28} (1966), 317--329.

\bibitem{haught-1986}
C.~A. Haught, \emph{The degrees below a 1-generic degree $< 0'$}, The Journal
  of Symbolic Logic \textbf{51} (1986), no.~3, 770--777.

\bibitem{hinman-2012}
P.~G. Hinman, \emph{A survey of {M}u{\v c}nik and {M}edvedev degrees}, The
  Bulletin of Symbolic Logic \textbf{18} (2012), no.~2, 161--229.

\bibitem{jockusch-1980}
C.~Jockusch, \emph{Degrees of generic sets}, Recursion theory: its
  generalisations and applications (F.~R. Drake and S.~S. Wainer, eds.), London
  Mathematical Society Lecture Note Series, Cambridge University Press, 1980,
  pp.~110--139.

\bibitem{kjoshanssen-nies-stephan-2005}
B.~Kjos-Hanssen, A.~Nies, and F.~Stephan, \emph{Lowness for the class of
  Schnorr random sets}, SIAM Journal on Computing \textbf{35} (2005), 647--657.

\bibitem{kleene-1945}
S.~C. Kleene, \emph{On the interpretation of intuitionistic number theory}, The
  Journal of Symbolic Logic \textbf{10} (1945), no.~4, 109--124.

\bibitem{mckinsey-tarski-1946}
J.~C.~C. McKinsey and A.~Tarski, \emph{On closed elements in closure algebras},
  The Annals of Mathematics Second Series \textbf{47} (1946), no.~1, 122--162.

\bibitem{mckinsey-tarski-1948}
\bysame, \emph{Some theorems about the sentential calculi of {L}ewis and
  {H}eyting}, The Journal of Symbolic Logic \textbf{13} (1948), no.~1, 1--15.

\bibitem{medvedev-1955}
Yu.~T. Medvedev, \emph{Degrees of difficulty of the mass problems}, Doklady
  Akademii Nauk SSSR, (NS) \textbf{104} (1955), no.~4, 501--504.

\bibitem{miller-martin-1968}
W.~Miller and D.~A. Martin, \emph{The degrees of hyperimmune sets}, Zeitschrift
  f{\"u}r Mathematische Logik und Grundlagen der Mathematik \textbf{14} (1968),
  159--166.

\bibitem{muchnik-1963}
A.~A. Muchnik, \emph{On strong and weak reducibilities of algorithmic
  problems}, Sibirskii Matematicheskii Zhurnal \textbf{4} (1963), 1328--1341.

\bibitem{odifreddi-1989}
P.~Odifreddi, \emph{Classical recursion theory}, Studies in Logic and the
  Foundations of Mathematics, vol. 125, North-Holland, 1989.

\bibitem{odifreddi-1999}
\bysame, \emph{Classical recursion theory: {V}olume {II}}, Studies in Logic and
  the Foundations of Mathematics, vol. 143, North-Holland, 1999.

\bibitem{posner-robinson-1981}
D.~B. Posner and R.~W. Robinson, \emph{Degrees joining to $0'$}, The Journal of
  Symbolic Logic \textbf{46} (1981), no.~4, 714--722.

\bibitem{skvortsova-1988}
E.~Z. Skvortsova, \emph{A faithful interpretation of the intuitionistic
  propositional calculus by means of an initial segment of the {M}edvedev
  lattice}, Sibirskii Matematicheskii Zhurnal \textbf{29} (1988), no.~1,
  171--178.

\bibitem{slaman-solovay-1991}
T.~A. Slaman and R.~M. Solovay, \emph{When oracles do not help}, COLT '91:
  Proceedings of the fourth annual workshop on Computational learning theory
  (M.~K. Warmuth and L.~G. Valiant, eds.), Morgan Kaufmann Publishers Inc.,
  1991, pp.~379--383.

\bibitem{sorbi-1996}
A.~Sorbi, \emph{The {M}edvedev lattice of degrees of difficulty},
  Computability, Enumerability, Unsolvability: Directions in Recursion Theory
  (S.~B. Cooper, T.~A. Slaman, and S.~S. Wainer, eds.), London Mathematical
  Society Lecture Notes, vol. 224, Cambridge University Press, 1996,
  pp.~289--312.

\bibitem{sorbi-terwijn-2012}
A.~Sorbi and S.~A. Terwijn, \emph{Intuitionistic logic and {M}uchnik degrees},
  Algebra Universalis \textbf{67} (2012), 175--188.

\bibitem{terwijn-2006}
S.~A. Terwijn, \emph{The {M}edvedev lattice of computably closed sets}, Archive
  for Mathematical Logic \textbf{45} (2006), no.~2, 179--190.

\bibitem{terwijn-zambella-2001}
S.~A. Terwijn and D.~Zambella, \emph{Computational randomness and lowness}, The
  Journal of Symbolic Logic \textbf{66} (2001), no.~3, 1199--1205.

\bibitem{vanoosten-2008}
J.~van Oosten, \emph{Realizability: an introduction to its categorical side},
  Studies in Logic and the Foundations of Mathematics, vol. 152, Elsevier,
  2008.

\end{thebibliography}
\end{document}